\documentclass[amscd,amssymb,verbatim,11pt]{amsart}
\usepackage{epsfig,amssymb}
\usepackage{graphics,amssymb}
\usepackage{amsthm}
\usepackage{amsfonts}
\usepackage{latexsym}
\usepackage{epsf}

\theoremstyle{definition}

\theoremstyle{remark}

\begin{document}

\title[Valuations of ASM]{The $p$-adic valuations of sequences counting 
alternating sign matrices}

\author[X. Sun]{Xinfu Sun}
\address{Department of Mathematics,
Tulane University, New Orleans, LA 70118}
\email{xsun1@math.tulane.edu}

\author[V. Moll]{Victor H. Moll}
\address{Department of Mathematics,
Tulane University, New Orleans, LA 70118}
\email{vhm@math.tulane.edu}

\newtheorem{Definition}{\bf Definition}[section]
\newtheorem{Thm}[Definition]{\bf Theorem}
\newtheorem{Example}[Definition]{\bf Example}
\newtheorem{Lem}[Definition]{\bf Lemma}
\newtheorem{Note}[Definition]{\bf Note}
\newtheorem{Cor}[Definition]{\bf Corollary}
\newtheorem{Conj}[Definition]{\bf Conjecture}
\newtheorem{Prop}[Definition]{\bf Proposition}
\newtheorem{Problem}[Definition]{\bf Problem}
\numberwithin{equation}{section}

\newcommand{\seqnum}[1]{\href{http://www.research.att.com/cgi-bin/access.cgi/as/~njas/sequences/eisA.cgi?Anum=#1}{\underline{#1}}}



\begin{abstract}
The $p$-adic valuations of a sequence of integers counting alternating 
sign symmetric matrices is examined for $p=2$ and $3$. Symmetry properties of 
their graphis produce a new proof of the result that characterizes the indices
that yield an odd number of matrices.
\end{abstract}

\maketitle

\newcommand{\nn}{\nonumber}
\newcommand{\ba}{\begin{eqnarray}}
\newcommand{\ea}{\end{eqnarray}}
\newcommand{\ift}{\int_{0}^{\infty}}
\newcommand{\ifft}{\int_{- \infty}^{\infty}}
\newcommand{\no}{\noindent}
\newcommand{\realpart}{\mathop{\rm Re}\nolimits}
\newcommand{\imagpart}{\mathop{\rm Im}\nolimits}
\newcommand{\mup}[2]{\mu_{#1}({#2})}
\newcommand{\muc}[1]{\mup{3}{#1}}
\newcommand{\nuthree}[1]{\nu_{3}(T(#1))}
\newcommand{\Sc}[1]{S_{3}(#1)}

\section{Introduction} \label{sec-intro}
\setcounter{equation}{0}

The magnificent book {\em Proofs and Confirmations} by David Bressoud 
\cite{bressoud1} tells the story of the {\em Alternating Sign Matrix
 Conjecture}(ASM) 
and its proof. This remarkable result involves the counting functions 
\begin{equation}
T(n) = \prod_{j=0}^{n-1} \frac{(3j+1)!}{(n+j)!} 
\label{T-seq}
\end{equation}
\noindent
and
\begin{equation}
C(n) = \prod_{j=0}^{n-1} \frac{(3j+1)! (6j)! (2j)! }{(3j)! (4j+1)!(4j)!}.
\end{equation}
\noindent
The survey by Bressoud and Propp \cite{bressoud99} describes the 
mathematics underlying  this problem. 

The fact that these numbers are integers is a direct consequence of their 
appearance as counting sequences. Mills, Robbins and Rumsey \cite{mrr-1} 
conjectured that the number of $n \times n$ matrices whose entries are 
$-1, \, 0,$ or $1$, whose row and column sums are all $1$, and such that
in every row, and in every column the non-zero entries alternate in sign is
given by $T(n)$. The first proof of this ASM
conjecture was provided by D. Zeilberger \cite{zeilberger96}. This
proof had the
added feature of being {\em pre-refereed}. Its $76$ pages were subdivided by
the author who provided a tree structure for the proof. An army of
volunteers provided checks for each node in the tree. The request for
checkers can be read in
\begin{center}
{\tt http://www.math.rutgers.edu/~\,zeilberg/asm/CHECKING}
\end{center}

The question of integrality of quotients
of factorials, such as $T(n)$, has been considered by 
D. Cartwright and J. Kupka in
\cite{cart-kupka}.  \\

\begin{Thm}
\label{cart-kup}
Assume that for every integer $k \geq 2$ we have 
\begin{equation}
\sum_{i=1}^{m} \left\lfloor{{a_{i}}\over{k}}\right\rfloor  \leq 
\sum_{j=1}^{n} \left\lfloor{{b_{j}}\over{k}}\right\rfloor.
\end{equation}
\noindent
Then the ratio of
$\begin{displaystyle}\prod_{j=1}^{n} b_{j}! \end{displaystyle}$ to
$\begin{displaystyle}\prod_{i=1}^{m} a_{j}! \end{displaystyle}$ is
an integer.
\end{Thm}

The authors \cite{cart-kupka} use this result to prove
that $T(n)$ is an integer.  \\

Given an interesting sequence of integers, it is a natural question to 
explore the structure of their factorization into primes. This is measured
by the $p$-adic valuation of the elements of the sequence.

\begin{Definition}
\label{def-valp}
Given a prime $p$ and a positive
integer $x \neq 0$, write $x = p^{m}y$, with $y$ not
divisible by $p$. The exponent $m$ is the $p$-adic valuation of $x$,
denoted by $m = \nu_{p}(x)$. This definition is extended to $x = a/b \in
\mathbb{Q}$ via $\nu_{p}(x) = \nu_{p}(a) - \nu_{p}(b)$. We leave the value
$\nu_{p}(0)$ as undefined.
\end{Definition}

The reader will find
in \cite{amm1} an analysis of the sequence 
\begin{equation}
A_{l,m} = \frac{l! m!}{2^{m-l}} \sum_{k=l}^{m} 2^{k} \binom{2m-2k}{m-k} \binom{m+k}{m} \binom{k}{l}
\end{equation}
\noindent
for fixed $l \in \mathbb{N}$. The sequence of rational numbers
\begin{equation}
d_{l,m} =  \frac{A_{l,m}}{l! m! 2^{m+l}}
\end{equation}
\noindent
appeared in \cite{bomohyper} in relation to the evaluation 
\begin{equation}
\int_{0}^{\infty} \frac{dx}{(x^{4}+2ax^{2} + 1)^{m+1}} = 
\frac{\pi}{\sqrt{2} m! (4(2a+1))^{m+1/2}} 
\sum_{l=0}^{m} A_{l,m} \frac{a^{l}}{l!}. 
\end{equation}
\noindent
This is a remarkable sequence of integers and some of its properties 
are described in \cite{manna-moll-survey}. In 
\cite{amm2} the reader will 
find similar studies for the Stirling numbers of the 
second kind. 

In this paper we discuss  the $p$-adic valuation of the sequence $T(n)$. 
The data seems erratic, as seen in the case of the first few primes 
\begin{eqnarray}
\nu_{2}(T(n)) & = & \{ 0, \, 1, \, 0, \, 1, \, 0, \, 2, \, 2, \, 3, \, 2, \, 
2, \, 0, \, 2, \, 2, \, 4, \, 4, \, 5, \, 4, \, 4, \, 2, \, 2, \, \cdots \}
\nonumber \\
\nu_{3}(T(n)) & = & \{ 0, \, 0, \, 0, \, 1, \, 1, \, 0, \, 0, \, 0, \, 0, \, 
1, \, 2, \, 3, \, 5, \, 5, \, 3, \, 2, \, 1, \, 0, \, 0, \, 0, \, \cdots \}.
\nonumber \\
\nu_{5}(T(n)) & = & \{ 0, \, 0, \, 0, \, 0, \, 0, \, 0, \, 0, \, 0, \, 1, \, 
2, \, 3, \, 4, \, 4, \, 3, \, 2, \, 1, \, 0, \, 0, \, 0, \, 0, \, \cdots \}.
\nonumber 
\end{eqnarray}

The goal of this paper is to provide a complete description of the 
function $\nu_{p}(T(n))$ for the primes $p=2$ 
and $p=3$. The case $p \geq 5$ presents similar features and the techniques
described here might be used to explain the graphs shown in Figure 
\ref{figure-5prime} and \ref{figure-7prime}. 
A detailed study of the graph of $\nu_{2} \circ T$ yields a new 
proof of a result of D. Frey and 
J. Sellers: the number $T(n)$ is odd if and only if $n$ is a 
{\em Jacobstahl number} $J_{m}$. These numbers are defined by the recurrence 
$J_{n} = J_{n-1} + 2 J_{n-2}$ with initial conditions $J_{1} = 1$ and 
$J_{2} = 3$. The proof presented here is based on the fact that the  
graph of $\nu_{2}(T(j))$ is formed by
{\em blocks} over the 
intervals $\{ \, [J_{n}, \, J_{n+1} ]: \, n \in \mathbb{N} \}$. Moreover, 
the part over $[J_{n+1}, \, J_{n} ]$ contains, at the center, a vertical shift
of the graph over $[J_{n-1}, \, J_{n}]$. This proves that the valuation 
$\nu \circ T$ can only vanish at the endpoints $J_{n}$.

Introduce a generalization of $T(n)$ as 
\begin{equation}
T_{p}(n) := \prod_{j=0}^{n-1} \frac{(pj+1)!}{(n+j)!}.
\end{equation}
\noindent
We will establish that, for each $p$, the numbers $T_{p}(n)$ are integers
and examine some of their divisibility properties. A combinatorial
interpretation of $T_{p}(n)$ is left as an open question. 

\section{A recurrence} \label{sec-rec}
\setcounter{equation}{0}

The integers $T(n)$ grow rapidly and a direct calculation using (\ref{T-seq})
is impractical. The number of digits of $T(10^{k})$ is 
$12, \, 1136, \, 113622$ and $11362189$ for 
$1 \leq k \leq 4$. Naturally, the prime factorization of $T(n)$ is more
promising, since every prime $p$ dividing $T(n)$ satisfies $p \leq 3n-2$. \\

In this section we discuss a recurrence for the $p$-adic valuation of 
$T(n)$, that permits a fast computation of this function. The statement
involves the function 
\begin{equation}
f_{p}(j) := \nu_{p}(j!).
\end{equation}

\begin{Thm}
\label{thm-recurr1}
Let $p$ be a prime. Then the $p$-adic valuation of $T(n)$ satisfies 
\begin{equation}
\nu_{p}(T(n+1)) = \nu_{p}(T(n)) + f_{p}(3n+1) + f_{p}(n) -f_{p}(2n) 
-f_{p}(2n+1).
\label{rec-1}
\end{equation}
\end{Thm}
\begin{proof}
This follows directly from comparing the expression 
\begin{equation}
\nu_{p}(T(n)) = \sum_{j=0}^{n-1} f_{p}(3j+1) - \sum_{j=0}^{n-1} f_{p}(n+j)
\end{equation}
\noindent
with the corresponding one for $\nu_{p}(T(n+1))$ and the initial value 
$T(1) = 1$. 
\end{proof}

Legendre \cite{legendre1} established the formula
\begin{equation}
f_{p}(j) = \nu_{p}(j!) = \frac{j - S_{p}(j)}{p-1},
\end{equation}
\noindent
where $S_{p}(j)$ denotes the sum of the base-$p$ digits of $j$. The result of 
Theorem \ref{thm-recurr1} is now expressed in terms of the function $S_{p}$.

\begin{Cor}
\label{p-val}
The $p$-adic valuation of $T(n)$ is given by 
\begin{equation}
\nu_{p}(T(n)) = \frac{1}{p-1} \left( \sum_{j=0}^{n-1} S_{p}(n+j) - 
\sum_{j=0}^{n-1} S_{p}(3j+1) \right). 
\end{equation}
\end{Cor}

Summing the recurrence (\ref{rec-1}) and using $T(1) = 1$ we obtain 
an alternative expression for the $p$-adic valuation of $T(n)$.

\begin{Prop}
\label{prop-valuetn}
The $p$-adic valuation of $T(n)$ is given by 
\begin{equation}
\nu_{p}(T(n)) = \frac{1}{p-1} \sum_{j=1}^{n-1} 
\left( S_{p}(2j) + S_{p}(2j+1) - S_{p}(3j+1) - S_{p}(j) \right). 
\label{valp-form1}
\end{equation}
\noindent
In particular, for $p=2$ we have 
\begin{eqnarray}
\nu_{2}(T(n)) & = & \sum_{j=0}^{n-1} 
\left( S_{2}(2j+1) - S_{2}(3j+1)  \right) \label{val-2} \\
& = & \sum_{j=1}^{n} 
\left( S_{2}(2j-1) - S_{2}(3j-2)  \right).
\nonumber
\end{eqnarray}
\end{Prop}

\begin{Cor}
For each $n \in \mathbb{N}$ we have
\begin{equation}
\sum_{j=1}^{n-1} S_{2}(2j+1) \geq \sum_{j=1}^{n-1} S_{2}(3j+1). 
\end{equation}
\end{Cor}

\noindent
{\bf Note}. The formula (\ref{valp-form1}) can be used to compute $T(n)$ for 
large values of $n$. Recall that only primes $p \leq 3n-2$ 
appear in the factorization of $T(n)$. For example, the number $T(100)$ has
$1136$ digits and its prime factorization is given by
\begin{multline}
T(100) = 2^{23} \cdot 3^{19} \cdot 13^{13} \cdot 17^{4} \cdot 29^{3} \cdot 41^{4} \cdot 61^{2} \cdot 67^{11} \cdot 71^{5}\cdot 
73^{3} \cdot 151 \cdot 157^{5} \cdot 163^{9} \cdot 167^{11} \\
\times 173^{15} 
\cdot 179^{19} \cdot 181^{21} \cdot 191^{27} \cdot 193^{29} \cdot 197^{31}
\cdot 199^{33} \cdot 211^{30} \cdot 223^{26} \cdot 227^{24} \cdot 229^{24}
\cdot 233^{22} \\
 \times 239^{20} \cdot 241^{40} \cdot 251^{16} \cdot 257^{14} 
\cdot 263^{12} \cdot 269^{10} \cdot 271^{10} \cdot 277^{8} \cdot 281^{6}
\cdot 283^{6} \cdot 293^{2}. \nonumber
\end{multline}

The recurrence (\ref{rec-1}) could be employed to generate large amount 
of data related to number theoretical questions 
associated to $T(n)$. In this paper we address the simplest of all: 
{\em characterize those indices} $n$ {\em for which } $T(n)$ {\em is odd}. \\

\section{When is $T(n)$ odd?} \label{sec-odd}
\setcounter{equation}{0}

Figure \ref{figure-1} shows the $2$-adic valuation of the sequence $T(n)$ for
$1 \leq n \leq 10^{5}$. Observe that $\nu_{2}(T(n)) \geq 0$ in view of the
fact that $T(n) \in \mathbb{N}$. Moreover, we see that $\nu_{2}(T(n)) = 0$
for a sequence of values starting  with
\begin{equation}
1, \, 3, \, 5, \, 11, \, 21, \, 43, \, 85, \, 171, \, 341, \, 683.
\end{equation}

\medskip

{{
\begin{figure}[ht]
\begin{center}
\centerline{\epsfig{file=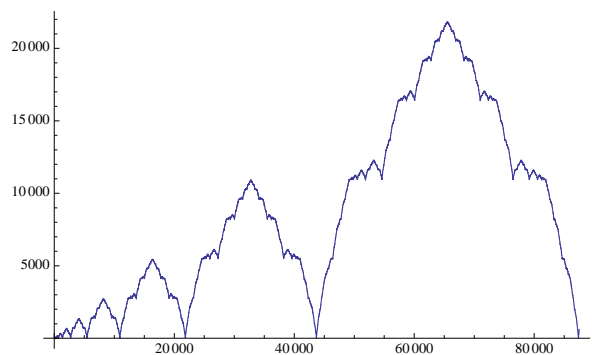,width=20em,angle=0}}
\caption{The $2$-adic valuation of $T(n)$}
\label{figure-1}
\end{center}
\end{figure}
}}

\medskip

A search in {\em The On-Line Encyclopedia of Integer Sequences} identifies 
these numbers as terms in the  {\em Jacobsthal sequence} (A001045),
 defined by the recurrence 
\begin{equation}
J_{n} = J_{n-1} + 2J_{n-2}, \text{ with } J_{0} = 1, J_{1} = 1.
\end{equation}
\noindent
The empirical observation  is that the sequence $T(n)$ is odd if and only if
$n$ is a Jacobsthal number; i.e., $n = J_{m}$ for some $m$.  \\

\noindent
{\bf Note}. The Jacobsthal numbers have many interpretations. 
Here is a small sample: \\

\noindent
a) $J_{n}$ is the numerator of the reduced fraction in the alternating 
sum 
$$\sum_{j=1}^{n+1} \frac{(-1)^{j+1}}{2^{j}}. $$ 

\noindent
b) Number of permutations with no fixed points avoiding $231$ and $132$. \\

\noindent
c) The number of odd coefficients in the expansion of 
$(1+x+x^{2})^{2^{n-1}-1}$.  \\

Many other examples can be found at

\begin{center}
\tt{http://www.research.att.com/\~\,njas/sequences/A001045}
\end{center}

\medskip

In this section we present a new proof of the following result
\cite{frey-sellers1}.  \\

\begin{Thm}
\label{thm-odd-0}
The number $T(n)$ is odd if and only if $n$ is a Jacobstahl number. 
\end{Thm}

The proof will employ several elementary
properties of the Jacobsthal number $J_{n}$, summarized here for the
convenience of the reader.
\begin{equation}
J_{n} = J_{n-1} + 2 J_{n-2}, \text{ with } J_{0}=1, \, J_{1} = 1. 
\end{equation}

\begin{Lem}
\label{jacob-prop}
For $n \geq 2$, the Jacobstahl numbers $J_{n}$ satisfy \\

\noindent
a) $ J_{n}= J_{n-1} + 2 J_{n-2}$ with $J_{0}=1$ and $J_{1}=1$. (This is 
the definition of $J_{n}$). \\ 

\noindent
b) $J_{n} = \frac{1}{3} ( 2^{n+1} + (-1)^{n} )$. \\

\noindent
c) $2^{n-1} + 1 \leq J_{n} < 2^{n}$. \\

\noindent
d) $J_{n}+J_{n-1} = 2^{n}$. \\

\noindent
e) $J_{n}-J_{n-2} = 2^{n-1}$.
\end{Lem}

\medskip

\noindent
{\bf Outline of the proof of Theorem \ref{thm-odd-0}}. The argument
is based on some observations from the
graph of the function $\nu_{2} \circ T$ as seen in Figure \ref{figure-1}.
The proof is divided into a small number
of steps, each one verified by an inductive procedure.  The hypothesis
assumes complete knowledge of the function $\nu_{2}(T(j))$
for $0 \leq j \leq J_{n}$.
We now show how to describe the function $\nu_{2} \circ T$ in the interval
$[J_{n}, \, J_{n+1}].$ \\

\noindent
{\bf Step 1}. The midpoint of the interval is $j=2^{n}$. The value there is
$\nu_{2}(T(2^{n})) = J_{n-1}$. This is Theorem \ref{thm-power}. \\

\noindent
{\bf Step 2}. The value $T(J_{n})$ is odd, that is, $\nu_{2}(T(J_{n})) = 0.$
This is the content of Theorem \ref{thm-odd}. \\

\noindent
{\bf Step 3}. Let $0 \leq i \leq 2J_{n-3}$. Then
\begin{equation}
\nu_{2}( T(J_{n} + i) ) = i + \nu_{2}( T(J_{n-2} + i )).
\end{equation}
\noindent
This is Lemma \ref{lemma-shift}. It describes the
function $\nu_{2} \circ T$ in the interval
$[J_{n}, \, 2^{n} - J_{n-2} ]$.  In particular, $\nu_{2}(T(2^{n}-J_{n-2}))
= 2J_{n-3}$ and $\nu_{2}(T(j)) > 0$ for $J_{n} < j < 2^{n} - J_{n-2}$. \\

\noindent
{\bf Step 4}. Let $0 \leq i \leq 2J_{n-2}$. Then
\begin{equation}
\nu_{2}(T(2^{n} - J_{n-2} + i)) = 
\nu_{2}(T(J_{n-1} + i))  + 2J_{n-3}. 
\end{equation}
\noindent
This is Proposition \ref{prop-shift}. It shows that
the graph of $\nu_{2} \circ T$ on the interval
$[ 2^{n} - J_{n-2}, 2^{n} + J_{n-2} ]$ is a vertical shift, by $2J_{n-3}$, of
the graph over the interval $[ J_{n-1}, J_{n} ]$. \\

\noindent
{\bf Step 5}. This is Proposition \ref{prop-symmetry}. Let
$0 \leq i \leq J_{n-1}$. Then 
$\nu_{2}(T(2^{n}-i)) = \nu_{2}(T(2^{n} + i ))$, explaining the symmetry of the
graph about the point $j = 2^{n}$ on the interval
$[J_{n}, J_{n+1}]$.  \\

This completes the proof of Theorem \ref{thm-odd-0}. \\

\medskip

{{
\begin{figure}[ht]
\begin{center}
\centerline{\epsfig{file=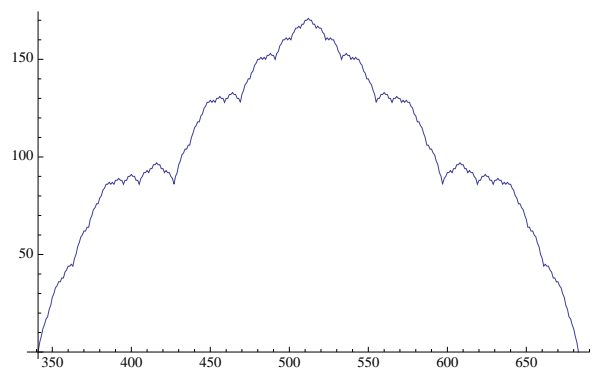,width=20em,angle=0}}
\caption{The $2$-adic valuation of $T(n)$ between minima}
\label{figure-3}
\end{center}
\end{figure}
}}

\medskip

\noindent
{\bf Note}. As we vary $m \in \mathbb{N}$, the graph of
$\nu_{2}(T(n))$ in the interval $[J_{m}, J_{m+1}]$
resemble each other. These are depicted in 
Figure \ref{figure-3} that shows the value of $\nu_{2}(T(n))$ for $J_{10}=341 
\leq n \leq 683 = J_{11}$. This suggests a possible scaling law for the
graph of $\nu_{2} \circ T$. 
Figure \ref{figure-3a} shows the first $15$ such 
graphs, scaled to the unit square.  The 
convergence to a limiting curve is apparent. The properties 
of this curve will be 
explored in the future. \\

{{
\begin{figure}[ht]
\begin{center}
\centerline{\epsfig{file=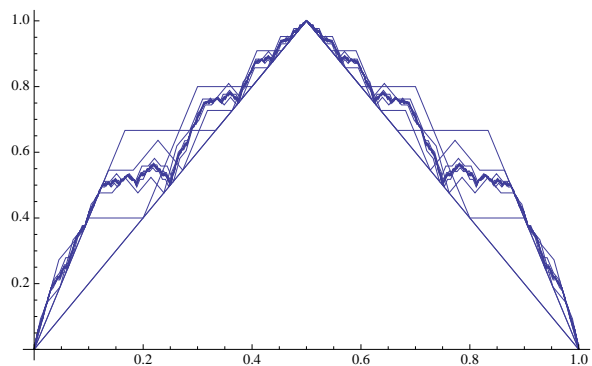,width=20em,angle=0}}
\caption{The scaled version of the $2$-adic valuation of $T(n)$}
\label{figure-3a}
\end{center}
\end{figure}
}}

\medskip

The proof of Theorem \ref{thm-odd-0} begins with an auxiliary lemma. \\

\begin{Lem}
\label{lemma-aux1}
Let $n \in \mathbb{N}$. Introduce the notation 
$S_{n,j}^{+} := S_{2}( 3 \cdot 2^{n} + 3j-2)$ and
$S_{n,j}^{-} := S_{2}( 3 \cdot 2^{n} - 3j+1 )$.
Then 
\begin{equation}
S_{n,j}^{+} = \begin{cases}
S_{2}(3j-2) + 2  & \text{ if } \quad
1 \leq j \leq J_{n-1},  \\
S_{2}(3j-2)  & \text{ if } \quad
1 + J_{n-1} \leq j \leq J_{n},  \\
S_{2}(3j-2) + 1  & \text{ if } \quad
1 + J_{n}  \leq j \leq 2^{n}; 
\end{cases}
\end{equation}
\noindent
and 
\begin{equation}
S_{n,j}^{-} = \begin{cases}
n+1-S_{2}(3j-2)  & \text{ if } \quad
1 \leq j \leq J_{n-1},  \\
n+2-S_{2}(3j-2)  & \text{ if } \quad
1 + J_{n-1} \leq j \leq J_{n},  \\
n+1-S_{2}(3j-2)  & \text{ if } \quad
1 + J_{n}  \leq j \leq 2^{n}.
\end{cases}
\end{equation}
\end{Lem}
\begin{proof}
Let $3j-2 = a_{0} + 2a_{1} + \cdots + a_{r}2^{r}$ be the binary expansion of 
$3j-2$. The corresponding one for $3 \cdot 2^{n-1}$ is simply 
$2^{n-1} + 2^{n}$. For $3j-2 < 2^{n-1}$ these two  expansions have no terms in
common, therefore $S_{n,j}^{+} = S_{2}(3j-2) + 2$. On the other hand, if
$2^{n-1} \leq 3j-2 < 2^{n}$ then the index in the binary expansion of $3j-2$
is $r = n-1$ with $a_{n-1} = 1$. The expansion  of $3j-2 + 3 \cdot 2^{n-1}$ is
now 
\begin{equation}
a_{0} + 2a_{1} + \cdots + a_{n-2}2^{n-2} + 2^{n-1} + 2^{n-1} + 2^{n} = 
a_{0} + 2a_{1} + \cdots + a_{n-2}2^{n-2} + 2^{n+1}, 
\nonumber
\end{equation}
\noindent
and this yields $S_{n,j}^{+} = a_{0} + a_{1} + \cdots + a_{n-2} + 1 = 
S_{2}(3j-2)$. The remaining cases are treated in a similar form. 
\end{proof}

We now establish the $2$-adic valuation at the center of the 
interval $[J_{n-1}, J_{n}]$. This completes Step 1 in the outline.

\begin{Thm}
\label{thm-power}
Let $n \in \mathbb{N}$. Then 
\begin{equation}
\nu_{2} \left( T \left( 2^{n} \right) \right) = J_{n-1}.
\end{equation}
\end{Thm}
\begin{proof}
We proceed by induction and split
\begin{equation}
\nu_{2} \left( T(2^{n}) \right) = 
\sum_{j=1}^{2^{n}-1} \left[ S_{2}(2j+1) - S_{2}(3j+1) \right]
\end{equation}
\noindent
at $j=2^{n-1}-1$. The first part is identified as 
$\nu_{2} \left( T( 2^{n-1} ) \right)$ to produce
\begin{equation}
\nu_{2} \left( T(2^{n}) \right) = 
\nu_{2} \left( T(2^{n-1}) \right)  +
\sum_{j=0}^{2^{n-1}-1} S_{2}(2j+1+ 2^{n}) 
- \sum_{j=1}^{2^{n-1}} S_{2}(3j-2 + 3 
\cdot 2^{n-1}).
\nonumber
\end{equation}
\noindent
Now observe that $2j+1 \leq 2^{n}-1 < 2^{n}$
so that $S_{2}(2j+1+2^{n}) = S_{2}(2j+1) + 1$.
Lemma \ref{lemma-aux1} gives, for $n$ even, 
\begin{multline}
\sum_{j=1}^{2^{n-1}} S_{2}( 3j-2 + 3 \cdot 2^{n-1} ) 
  =  
\sum_{j=1}^{(2^{n-1}+1)/3} [S_{2}(3j-2) + 2 ] + \\
\sum_{j=(2^{n-1}+1)/3}^{(2^{n}-1)/3} S_{2}(3j-2) + 
\sum_{j=(2^{n}+2)/3}^{2^{n-1}} [S_{2}(3j-2) + 1 ]  \nonumber
\end{multline}
\noindent
and using (\ref{val-2}) yields
\begin{equation}
\nu_{2}(T(2^{n})) = 2 \nu_{2}(T(2^{n-1})) - 1 = 2J_{n-2}-1. 
\end{equation}
\noindent
Elementary
properties of Jacobsthal numbers show that $2J_{n-2} - 1 = J_{n-1}$ proving
the result for $n$ even. The argument for $n$ odd is similar. 
\end{proof}

\medskip

The next theorem corresponds to Step 2 of the outline. \\

\begin{Thm}
\label{thm-odd}
Let $n \in \mathbb{N}$. Then
$T(J_{n})$ is odd.
\end{Thm}
\begin{proof}
Proposition \ref{prop-valuetn} gives
\begin{equation}
\nu_{2} \left( T(J_{n}) \right) = 
\sum_{j=1}^{J_{n}-1} \left[ S_{2}(2j+1) - S_{2}(3j+1) \right]. 
\end{equation}

\noindent
Observe that $2^{n-1} \leq J_{n} -1$, so 
\begin{eqnarray}
\nu_{2} \left( T(J_{n}) \right) & = & 
\sum_{j=1}^{2^{n-1}-1} \left[ S_{2}(2j+1) - S_{2}(3j+1) \right] +  \nonumber \\
&  &  +
\sum_{j=2^{n-1}}^{J_{n}-1} 
\left[ S_{2}(2j+1) - S_{2}(3j+1) \right]   \nonumber \\
& = & \nu_{2} \left( T ( 2^{n-1} ) \right) + 
\sum_{j=2^{n-1}}^{J_{n}-1} 
\left[ S_{2}(2j+1) - S_{2}(3j+1) \right].   \nonumber 
\end{eqnarray}

Therefore
\begin{equation}
\nu_{2} \left( T(J_{n}) \right)  =  \nu_{2}(T(2^{n-1}))  + 
\sum_{j=0}^{J_{n} - 1 - 2^{n-1}} 
\left[ S_{2} ( 2j+1 + 2^{n}) - S_{2}(3j+1 + 3 \cdot 2^{n-1})  \right].
\nonumber 
\end{equation}
The elementary properties of Jacobsthal numbers give
\begin{equation}
J_{n}-1 - 2^{n-1} = J_{n-2} - 1,
\end{equation}
\noindent
so that
\begin{equation}
\nu_{2} \left( T(J_{n}) \right)  =  \nu_{2}(T(2^{n-1}))   +
\sum_{j=0}^{J_{n-2} - 1} 
\left[ S_{2} ( 2j+1 + 2^{n}) - S_{2}(3j+1 + 3 \cdot 2^{n-1})  \right].
\nonumber
\end{equation}

\noindent
Observe that
\begin{equation}
2j+1 \leq 2(J_{n-2}-1)+1 = 2J_{n-2}-1 = J_{n}-J_{n-1}-1 < 2^{n},
\nonumber
\end{equation}
\noindent
resulting in
\begin{equation}
S_{2}(2j+1+2^{n}) = S_{2}(2j+1) + 1. 
\nonumber
\end{equation}
\noindent
Similarly 
$3j+1 \leq 3J_{n-2}-2 < 3 (2^{n-1}+(-1)^n) -2 \leq 2^{n-1} -1$ and 
from $3 \cdot 2^{n-1} = 2^{n} + 2^{n-1}$ we obtain
\begin{equation}
S_{2}(3j+1 + 3 \cdot 2^{n-1} ) = S_{2}(3j+1) + 2, 
\nonumber
\end{equation}
\noindent
for $0 \leq j \leq J_{n-2}-1$. It follows that
\begin{equation}
\nu_{2} \left( T(J_{n}) \right) = \nu_{2} \left( T(2^{n-1}) \right)  + 
\sum_{j=0}^{J_{n-2}-1} \left[ S_{2}(2j+1) - S_{2}(3j+1) \right] - J_{n-2}.
\nonumber
\end{equation}
\noindent
Theorem \ref{thm-power} shows that the first and third term on the line 
above cancel, leading to 
\begin{equation}
\nu_{2} \left( T(J_{n}) \right)  = 
\nu_{2} \left( T(J_{n-2}) \right). 
\nonumber
\end{equation}
\noindent
The result now follows by induction on $n$.
\end{proof}

\medskip

We continue with the proof of Theorem \ref{thm-odd-0}. The next Lemma
corresponds to Step 3 in the outline. It describes the values
$\nu_{2}(T(j))$ for $J_{n} \leq j \leq J_{n} + 2J_{n-3} = 2^{n} - J_{n-2}$.
The result of Lemma \ref{lemma-shift} shows that
$\nu_{2}(T(j)) > 0$ for $J_{n} < j < 2^{n} - J_{n-2}$. \\

\begin{Lem}
\label{lemma-shift}
For $0 < i \leq 2J_{n-3}$ we have 
\begin{equation}
\nu_{2}(T(J_{n}+i)) = i + \nu_{2}(T(J_{n-2}+i)).
\end{equation}
\end{Lem}

\begin{proof}
Assume that $n$ is even and consider
\begin{eqnarray} 
\nu_{2}(T(J_{n}+i)) & = & 
\sum_{j=1}^{J_{n}+i-1} 
\left[ S_{2}(2j+1) - S_{2}(3j+1) \right] \nonumber \\
& = & 
\sum_{j=1}^{J_{n}-1} 
\left[ S_{2}(2j+1) - S_{2}(3j+1) \right] + 
\sum_{j=J_{n}}^{J_{n}+i-1} 
\left[ S_{2}(2j+1) - S_{2}(3j+1) \right]. \nonumber
\end{eqnarray}
\noindent
The first sum is $\nu_{2}(T(J_{n})) = 0$, according to Theorem \ref{thm-odd}. 
Therefore, using Lemma \ref{jacob-prop} we have
\begin{eqnarray} 
\nu_{2}(T(J_{n}+i)) & = & 
\sum_{j=J_{n}}^{J_{n}+i-1} 
\left[ S_{2}(2j+1) - S_{2}(3j+1) \right] \nonumber \\
& = & \sum_{j=J_{n}+1}^{J_{n}+i} 
\left[ S_{2}(2j-1) - S_{2}(3j-2) \right]  \nonumber  \\
& = & \sum_{j=J_{n}+1-2^{n-1}}^{J_{n}+i-2^{n-1}} 
\left[ S_{2}(2^{n}+2j-1) - S_{2}(3 \cdot 2^{n-1} + 3j-2) \right]  \nonumber 
\\
 & =  & \sum_{j=J_{n-2}+1}^{J_{n-2}+i} 
\left[ S_{2}(2^{n}+2j-1) - S_{2}(3 \cdot 2^{n-1} + 3j-2) \right].   \nonumber
\end{eqnarray}
\noindent
The index $j$ satisfies 
\begin{equation}
2j-1 \leq 2(J_{n-2}+i)-1 < 2(J_{n-2}+2J_{n-3}) = 2J_{n-1} < 2^{n},
\nonumber
\end{equation}
\noindent
therefore $S_{2}(2^{n}+2j-1) = 1 + S_{2}(2j-1)$. 

The lower limit in the last sum is 
$J_{n-2} + 1 = \frac{1}{3}(2^{n-1}+1) +1$, and the upper bound is 
\begin{equation}
J_{n-2}+i \leq J_{n-2} + 2J_{n-3} = J_{n-1} = 
\frac{1}{3}(2^{n}-1). 
\end{equation}
\noindent
Lemma \ref{lemma-aux1} gives $S_{2}(3 \cdot 2^{n-1} + 3j-2) = 
S_{2}(3j-2)$. Therefore 
\begin{eqnarray} 
\nu_{2}(T(J_{n}+i)) 
&  =  & \sum_{j=J_{n-2}+1}^{J_{n-2}+i} 
\left[ S_{2}(2j-1) + 1- S_{2}(3j-2) \right] \nonumber \\
&  =  & i + \sum_{j=J_{n-2}+1}^{J_{n-2}+i} 
\left[ S_{2}(2j-1) - S_{2}(3j-2) \right] \nonumber \\
& = & i + \nu_{2}(T(J_{n-2}+i)). \nonumber
\end{eqnarray}
The result has been established for $n$ even. The proof for $n$ odd is
similar. 
\end{proof}

\medskip

The next result shows the graph of $\nu_{2} \circ T$ on the 
interval $[2^{n}-J_{n-2},2^{n}+J_{n-2}]$ is a vertical shift of the 
graph on $[J_{n-1},J_{n}]$.  This corresponds to Step 4 in the outline. \\

\begin{Prop}
\label{prop-shift}
For $0 \leq i \leq 2J_{n-2}$, 
\begin{equation}
\nu_{2}(T(2^{n}-J_{n-2}+i)) = 
\nu_{2}(T(J_{n-1}+i)) + \omega_{n},
\end{equation}
\noindent
where $\omega_{n} = 2J_{n-3}$ is independent of $i$. 
\end{Prop}
\begin{proof}
We prove that the graph of $\nu_{2}(T(J_{n-1} + i))$  and 
$\nu_{2}(T(2^{n}-J_{n-2} + i))$ have the same discrete derivative.   This
amounts to checking the identity
\begin{multline}
\nu_{2}(T(J_{n-1}+i)) - 
\nu_{2}(T(J_{n-1}+i-1))  =  \\
\nu_{2}(T(2^{n}-J_{n-2}+i))  - 
\nu_{2}(T(2^{n}-J_{n-2}+i-1))
\end{multline}
\noindent
for $1 \leq i \leq 2J_{n-2}$. Observe that 
\begin{equation}
\nu_{2}(T(k)) - \nu_{2}(T(k-1)) = S_{2}(2k-1) - S_{2}(3k-2),
\end{equation}
\noindent
and using $2^{n}-J_{n-2} = 2^{n-1}+J_{n-1}$, we conclude that the 
result is equivalent to the identity 
\begin{multline}
S_{2}(2^{n} + 2(J_{n-1}+i)-1) - 
S_{2}(2(J_{n-1}+i)-1) = \\
S_{2}(3 \cdot 2^{n-1} + 3(J_{n-1}+i)-2)
- S_{2}(3(J_{n-1}+i)-2), \label{iden-1}
\end{multline}
\noindent
for $1 \leq i \leq 2J_{n-2}$. Define 
\begin{equation}
h_{n}(i) = \begin{cases}
             1 \quad & \text{ if } 1 \leq i \leq J_{n-2}; \\
             0 \quad & \text{ if } J_{n-2} + 1 \leq i \leq 2J_{n-2}.
 \end{cases}
\end{equation}
\noindent
The assertion is that both sides in (\ref{iden-1})  agree with $h_{n}(i)$. 
The analysis of the left hand side is easy: the condition 
$1 \leq i \leq J_{n-2}$ implies $2(J_{n-1}+i)-1 \leq 2^{n}-1$. Thus,
the term $2^{n}$ does not interact with the binary expansion 
$2(J_{n-1}+i)-1$ and produces 
the extra $1$. On the other hand, if $J_{n-2}+ 1 \leq i \leq 2 J_{n-2}$, then 
\begin{multline}
2^{n}+1  = 2(J_{n-1}+J_{n-2}+1) -1 \leq 2(J_{n-1}+i)-1 \\
\leq 2(J_{n-1}+2J_{n-2})-1 = 2J_{n}-1 < 2^{n+1}-1.
\end{multline}
\noindent
We conclude that the binary expansion of $x := 2(J_{n-1}+i)-1$ is of the 
form $a_{0}+a_{1}\cdot 2 + \cdots + a_{n-1} \cdot 2^{n-1} + 1 \cdot 2^{n}$.
It follows that $2^{n} + x$ and $x$ have the same number of $1$'s in their
binary expansion. Thus
$S_{2}(x) = S_{2}(x+2^{n})$ as claimed.  \\

The analysis of the right hand side of (\ref{iden-1}) is slightly more 
difficult. Let $x := 3(J_{n-1}+i)-2$ and it is required to compare 
$S_{2}(x)$ and $S_{2}(3 \cdot 2^{n-1} + x)$. Observe that 
\begin{equation}
x \leq 3(J_{n-1} + 2J_{n-2} ) -2 = 3J_{n}-2 = 2^{n+1} + (-1)^{n} - 2 < 2^{n+1}
\end{equation}
\noindent
and 
\begin{equation}
x \geq 3(J_{n-1}+1) -2 = 2^{n} + (-1)^{n-1} +1 \geq 2^{n}.
\end{equation}
\noindent
We conclude that the binary expansion of $x$ is of the form 
\begin{equation}
x = a_{0} + a_{1} \cdot 2 + \cdot + a_{n-1} \cdot 2^{n-1} + 1 \cdot 2^{n},
\end{equation}
\noindent 
and the corresponding one for $3 \cdot 2^{n-1}$ is $2^{n} + 2^{n-1}$. An 
elementary calculation shows that 
$S_{2}(x + 3 \cdot 2^{n-1} ) - S_{2}(x)$ is $1$ if $a_{n-1}=0$ and $0$ if 
$a_{n-1}=1$. In order to transform this inequality to a restriction on the
index $i$, observe that $a_{n-1}=1$ is equivalent to $x - 2^{n} \geq 2^{n-1}$.
Using the value of $x$ this becomes $3(J_{n-1}+i)-2) \geq 3 \cdot 2^{n-1}$.
This is directly transformed to $i \geq J_{n-2}+1$. This shows that the 
right hand side of (\ref{iden-1}) also agrees with $h_{n}$ and (\ref{iden-1})
has been established. 
\end{proof}

\medskip

The final step in the proof of Theorem \ref{thm-odd-0}, outlined as Step 5, 
shows the symmetry of the graph 
of $\nu_{2}(T(j))$ about the point $j = 2^{n}$. The range covered in the 
next proposition is $2^{n}-J_{n-1} \leq j \leq 2^{n}+J_{n-1}$. \\

\begin{Prop}
\label{prop-symmetry}
For $1 \leq i \leq J_{n-1}$,
\begin{equation}
\nu_{2}(T(2^{n}-i)) = \nu_{2}(T(2^{n}+i)). 
\end{equation}
\end{Prop}
\begin{proof}
Start with 
\begin{eqnarray}
\nu_{2}(T(2^{n})) - \nu_{2}(T(2^{n}-i)) & = & 
\sum_{j=2^{n}-i+1}^{2^{n}} \left[ S_{2}(2j-1) - S_{2}(3j-2) \right] 
\nonumber \\
& = & \sum_{k=1}^{i} \left[ S_{2}(2^{n+1}-(2k-1)) - 
S_{2}( 3 \cdot 2^{n} -(3k-1)) \right]. 
\nonumber
\end{eqnarray}
\noindent
The first term in the sum satisfies 
\begin{equation}
S_{2}(2^{n+1} -(2k-1) ) = n+2 - S_{2}(2k-1).
\end{equation}
\noindent
To check this, write $2k-1 = a_{0} + a_{1} \cdot 2 + \cdots + a_{r} 
\cdot 2^{r}$ with $a_{0} = 1$ because $2k-1$ is odd. Now, 
$2^{n+1} = (1 + 2 + 2^{2} + \cdots + 2^{n} ) + 1$ and we conclude that
\begin{eqnarray}
2^{n+1} - (2k-1) & = & ( 2^{n} + 2^{n-1} + \cdots + 2^{r+1} ) \nonumber \\
 &   & + (1-a_{r}) \cdot 2^{r} + (1-a_{r+1}) \cdot 2^{r-1} + \cdots + 
(1-a_{1}) \cdot 2 + 1  \nonumber 
\end{eqnarray}
\noindent
Therefore
\begin{eqnarray}
S_{2}( 2^{n+1} - (2k-1))  
& =  & n+1 - \left( a_{r}+a_{r-1} + \cdots + a_{1} \right)  \nonumber  \\
& =  & n+2 - S_{2}(2k-1). \nonumber  
\end{eqnarray}

We conclude that
\begin{multline}
\nu_{2}(T(2^{n})) - \nu_{2}(T(2^{n}-i)) = 
(n+2)i - \sum_{k=1}^{i} S_{2}(2k-1) -  \\
\sum_{k=1}^{i} S_{2}( 3 \cdot 2^{n} - (3k-1)). 
\end{multline}

Similarly
\begin{eqnarray}
\nu_{2}(T(2^{n}+i)) - \nu_{2}(T(2^{n}))  & = & 
\sum_{j=2^{n}+1}^{2^{n}+i} \left( S_{2}(2j-1) - S_{2}(3j-2) \right) 
\nonumber \\
& =  & \sum_{k=1}^{i} \left( S_{2}(2^{n+1}+2k-1) - S_{2}(3 \cdot 2^{n} + 3k-2 ) 
\right). \nonumber
\end{eqnarray}
\noindent
The inequality
\begin{equation}
2k-1 \leq 2i-1 \leq 2J_{n-1}-1 \leq 2 \cdot 2^{n-1} -1 
\leq 2^{n}-1 < 2^{n+1}
\end{equation}
\noindent
shows that $S_{2}(2^{n+1} + 2k-1 ) = 1 + S_{2}(2k-1)$. Lemma \ref{lemma-aux1}
yields the identity
\begin{equation}
S_{2}(3 \cdot 2^{n} + 3k-2 ) + S_{2}(3 \cdot 2^{n} -3k+1) = n+3.
\label{ident-symm}
\end{equation}
\noindent
Therefore
\begin{eqnarray}
\nu_{2}(T(2^{n}+i)) - \nu_{2}(T(2^{n}))  & = & 
\sum_{k=1}^{i} \left( S_{2}(2^{n+1}+2k-1) - S_{2}(3 \cdot 2^{n} + 3k-2 ) 
\right) + i \nonumber \\
& & +  \sum_{k=1}^{i} S_{2}(2k-1) - 
\left( n+3 - S_{2}(3 \cdot 2^{n} -3k+1) \right). \nonumber
\end{eqnarray}
\noindent
It follows that 
\begin{equation}
\nu_{2}(T(2^{n})) - \nu_{2}( T(2^{n}-i)) = 
- \left[ \nu_{2}( T(2^{n}-i))  - \nu_{2}(T(2^{n})) \right],
\nonumber
\end{equation}
\noindent
and symmetry has been established.
\end{proof}

\noindent
{\bf Note}. The identity (\ref{ident-symm}) can be given a direct proof by 
inducting on $k$. It 
is required to check that the left hand side is independent of $k$ and this
follows from the identity
\begin{equation}
S_{2}(m+3) - S_{2}(m) = 
\begin{cases}
2 - \omega_{2} \left( \frac{m}{2} \right) \quad  \text{ if } 
m \equiv 0 \bmod 2; \\
 - \omega_{2} \left( \lfloor{ {{m} \over {4}} \rfloor}\right)  
\quad  \text{ if } 
m \equiv 1 \bmod 2.
\end{cases}
\end{equation}
\noindent
Here $\omega_{2}(m)$ is the number of trailing $1$'s in the binary expansion 
of $m$. For $m= 829$ we have $S_{3}(829) = 7$ and $S_{3}(832) = 3$. The 
binary expansion of $m=207 = \lfloor{829/4 \rfloor}$ is $11001111$ and the 
number of trailing $1's$ is $4$. This 
observation is due to A. Straub. 

\medskip

The next result shows that every positive integer $k$ is attained as
$\nu_{2}(T(n))$. 

\begin{Thm}
Every nonnegative integer appears as $\nu_2(T(n))$ for some $n$, i.e.,
$$\mathbb{N} = \{\nu_2(T(n)) \, :  n \in \mathbb{N} \}.$$
Furthermore, each positive integer $m$ appears only finitely many times, 
and the last appearance is when $n = J_{2m+1}-1$.
\end{Thm}
\begin{proof}
From the results before, we know that 
\begin{equation}
\nu_2(T(J_n+i)) > \nu_2(T(J_n+1)) 
= \nu_2(T(J_{n+1}-1)), \nonumber
\end{equation}
\noindent
for $1 < i < J_{n+1}-J_n-2$ and
$\nu_2(T(J_{n+2}-1)) = \nu_2(T(J_n-1)) + 1$. This shows that 
the minimum values of the graph of $\nu_2(T(n))$ around $2^n$
are attained exactly at $J_n+1$ and $J_{n+1}-1$.
These values are also {\em strictly} increasing along the even and odd
indices. Thus, $m < \nu_{2}(T(i))$ 
for any given $m$, provided $i$ is large enough.

To determine the last appearance of $m$, we only need to 
determine the last occurance of $n$ 
such that $\nu_2(T(J_n-1)) = m$. Since
$\nu_2(T(J_2-1)) = \nu_2(T(J_3-1)) = 1,$ we conclude 
that $\nu_2(T(J_{2n}-1))=\nu_2(T(J_{2{n+1}}-1))=n$.
Therefore the last occurance for $m$ is at $J_{2m+1}-1$.
\end{proof}

\noindent
{\bf Note}. Define $\lambda(m)$ to be the number $m$ 
is attained by $\nu \circ T$. The values for $1 \leq m \leq 8$ are
shown below.

\begin{table}[h]
\begin{center}
\begin{tabular}{||c||c|c|c|c|c|c|c|c||}
\hline 
$m $& 1 & 2 & 3 & 4 & 5 & 6 & 7 & 8   \\ \hline 
$\lambda(m)$ & 2 & 8 & 5 & 12 & 5 & 14 & 8 & 14  \\
\hline
\end{tabular}
\end{center}
\vskip 0.2in
\caption{The first $8$ values in the range of $\nu_{2} \circ T$} 
\end{table}

\noindent
For example, the values of $n$ for which $\nu(T(n))=5$ are 
$16, \, 342, \, 682, \, 684$ and 
$J_{11} -1  = 1364$ and the eight solutions to $\nu(T(n))=7$ are
$26, \, 38, \, 46, \, 82, \, 5462,
\newline 
\, 10922, \, 10924$ and $J_{15}-1 = 21844$. \\

\noindent
{\bf Note}. In sharp contrast to the $2$-adic valuation, D. Frey 
and J. Sellers \cite{frey-sellers2, frey-sellers3}
show that if $p \geq 3$ is
a prime, then for each nonnegative integer $m$ there
exist infinitely many positive integers $n$ for which $\nu_{p}(T(n)) = m$.

\section{The $3$-adic valuation of $T(n)$} \label{sec-three}
\setcounter{equation}{0}

The analysis of the $2$-adic valuation of $T(n)$ is now extended to the 
prime $p=3$. The discussion employs the expansion of $n \in \mathbb{N}$
in base $3$, given by 
\begin{equation}
n = a_{0} + a_{1} \cdot 3 + a_{2} \cdot 3^{2} + \cdots + 
a_{r} \cdot 3^{r}
\label{base3-exp}
\end{equation}
\noindent
and the function
\begin{equation}
S_{3}(n) := a_{0} + a_{1} + \cdots + a_{r}. 
\label{sum-base3}
\end{equation}

{{
\begin{figure}[ht]
\begin{center}
\centerline{\epsfig{file=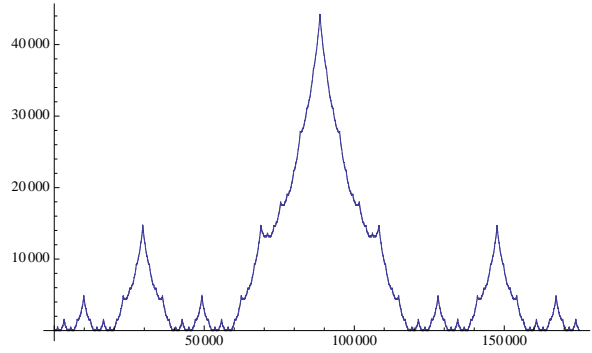,width=20em,angle=0}}
\caption{The $3$-adic valuation of $T(n)$}
\label{figure-2}
\end{center}
\end{figure}
}}

Figure \ref{figure-2} presents a well-defined symmetry 
for $\nu_{3}(T(n))$. This is explained in Theorem \ref{thm-3symmetry}. \\ 

The first result characterizes the values $n$ for which $\nu_{3}(T(n)) = 0$. 

\begin{Thm}
\label{val3-zero}
Let $n \in \mathbb{N}$ with (\ref{base3-exp}) as its expansion in base $3$. Then
$\nu_{3}(T(n)) = 0$ if and only if there is 
an index $0 \leq i \leq r$ such that $a_{0}=a_{1}= \cdots = a_{i-1} = 0$
and $a_{i+1}=a_{i+2} = \cdots = a_{r} = 0 \text{ or } 2$, with 
$a_{i}$ arbitrary.
\end{Thm}

We begin with some elementary results on the function $S_{3}$ 
which admit elementary proofs. 

\begin{Lem}
\label{3-rules}
Let $n \in \mathbb{N}$. Then 
\begin{equation}
S_{3}(3n) = S_{3}(n), \, S_{3}(3n+1) = 1 + S_{3}(n) \text{ and }
S_{3}(3n+2) = 2 + S_{3}(n).
\nonumber
\end{equation}
\end{Lem}

\begin{Lem}
\label{3-rules-a}
Let $n \in \mathbb{N}$. Then 
\begin{eqnarray}
\Sc{4 \cdot 3^n+b} & = & 2+\Sc{b} \text{ for all }  0 \le b < 2 \cdot 3^n, 
\nonumber \\
\Sc{2 \cdot 3^{n}+b} & = & 2+\Sc{b} \text{ for all } 0 \le b < 3^{n},  
\nonumber \\
\Sc{3^{n}+b-1} & = & 1+\Sc{b-1} \text{ for }1 \le b < 3^n.
\nonumber
\end{eqnarray}
\end{Lem}

The next step in analyzing the function $\nu_{3} \circ T$ is to produce a 
recurrence for this valuation. The 
symmetry observed in Figure \ref{figure-2} is a consequence of this result.

\begin{Prop}
\label{thm-3symmetry}
Let $n \in \mathbb{N}$. Then $\nu_{3}(T(3n)) = 3 \nu_{3}(T(n))$. 
\end{Prop}
\begin{proof}
Legendre's formula (\ref{p-val}) shows that the result is equivalent to 
\begin{equation}
\sum_{j=0}^{3n-1} S_{3}(3n+j) - \sum_{j=0}^{3n-1} S_{3}(3j+1) 
- 3 \sum_{j=0}^{n-1} S_{3}(n+j) + 3 \sum_{j=0}^{n-1} S_{3}(3j+1) = 0.
\label{mess}
\end{equation}
\noindent
Each term of (\ref{mess}) is now simplified. Lemma \ref{3-rules} shows that
\begin{eqnarray}
\sum_{j=0}^{3n-1} S_{3}(3n+j) & = & 
\sum_{j=0}^{n-1} S_{3}(3(n+j))  + 
\sum_{j=0}^{n-1} S_{3}(3(n+j)+1)  + 
\sum_{j=0}^{n-1} S_{3}(3(n+j)+2)  \nonumber \\
& = & 3n + 3 \sum_{j=0}^{n-1} S_{3}(n+j), \nonumber 
\end{eqnarray}
\noindent
and
\begin{eqnarray}
\sum_{j=0}^{3n-1} S_{3}(3j+1) & = & 3n + \sum_{j=0}^{3n-1} S_{3}(j) \nonumber \\
& = & 3n + \sum_{j=0}^{n-1} S_{3}(3j) + 
\sum_{j=0}^{n-1} S_{3}(3j+1) + 
\sum_{j=0}^{n-1} S_{3}(3j+2)  \nonumber \\
& = & 6n + 3 \sum_{j=0}^{n-1} S_{3}(j), \nonumber 
\end{eqnarray}
\noindent
and, finally,
\begin{eqnarray}
\sum_{j=0}^{n-1} S_{3}(3j+1) & = & n + \sum_{j=0}^{n-1} S_{3}(j).
\nonumber
\end{eqnarray}
\noindent
These identities show that the left-hand side of  (\ref{mess}) vanishes. 
\end{proof}

\begin{Cor}
\label{end-cases}
For each $n \in \mathbb{N}$, we have $\nuthree{3^n}=\nuthree{2 \cdot
3^{n}} = 0.$
\end{Cor}
\begin{proof}
This follows directly from $T(1) = 1$ and $T(2) = 1$ and Proposition
\ref{thm-3symmetry}. 
\end{proof}

For brevity, introduce the function
\begin{equation}
\muc{j} := S_3(2j)+S_3(2j+1)-S_3(3j+1)-S_3(j).
\end{equation}
\noindent
Thus Proposition \ref{prop-valuetn} takes the form
\begin{equation}
\nu_{3}(T(n)) = \tfrac{1}{2}  \sum_{j=1}^{n-1} \muc{j}. 
\label{nu3-tn}
\end{equation}
\noindent
Observe that 
\begin{equation}
\muc{n-1} = 2 \left( \nu_{3}(T(n)) - \nu_{3}(T(n-1)) \right).
\end{equation}

\begin{Prop}
\label{zeros-1}
If $0 \leq a \leq 3^n$ then $\nuthree{a}=\nuthree{2 \cdot 3^n+a}$.
\end{Prop}
\begin{proof}
The limiting cases $a=0$ and $a = 3^{n}$ follow from Corollary 
\ref{end-cases}. The result follows from (\ref{nu3-tn}) and the
identities
$\muc{a}=\muc{2 \cdot 3^n+a}$ for $ 1 \le a \le 3^n,$ that are direct
consequence of Lemma \ref{3-rules-a}.
\end{proof}

\noindent
The proof of Theorem \ref{val3-zero} is presented next. 

\begin{proof}
Consider the representation of 
$n \in \mathbb{N}$ in base $3$: 
\begin{equation}
n = a_{0} + 3a_{1} + 3^{2} a_{2} + \cdots + 3^{r}a_{r}.
\end{equation}
\noindent
Corollary \ref{end-cases} and Proposition \ref{zeros-1} show 
that the numbers $n$ with the form stated in 
the theorem satisfy 
$\nu_{3}(T(n)) = 0$. We need to prove that these are the only zeros of 
$\nu_{3} \circ T$. 

The proof is by induction and show that 
$\nuthree{a} > 0$ for $3^n < a < 3^{n+1}$. 
Proposition \ref{zeros-1} shows that, if $a_{r}=2$, then $\nu_{3}(T(n)) > 0$. 
Proposition \ref{prop-case1} treats the result for $a_{r}=1$ and the 
first half of these numbers $0 \leq a-3^{r} \leq 3^{r}$. Proposition 
\ref{prop-case2}  establishes a symmetry result that takes care of the 
second half.  \\
\end{proof}

\medskip

We now establish the symmetry of the function $\nu_{3} \circ T$. The 
proof begin with some auxiliary steps. \\

\begin{Prop}
\label{prop-case1}
Let $n, \, a  \in \mathbb{N}$ and assume $1 \leq a < 3^{n}$. Then 
\begin{equation}
\muc{3^n+a}=\left\{
\begin{array}{lll}
\muc{a}+2	&	\text{if } 1 \le a < \tfrac{1}{2}3^n;  \nonumber \\
\muc{a}		&	\text{if } a = \tfrac{1}{2}(3^n+1); \nonumber  \\
\muc{a}-2	&	\text{if } \tfrac{1}{2}3^n+1 < a \le 3^n. \nonumber
\end{array}
\right.
\end{equation}
\end{Prop}

\begin{proof}
When $1 \le b < \tfrac{1}{2}3^n$,
the first part  follows from Lemma \ref{3-rules-a}. 
The other parts can be proved similarly, and thus omitted.
\end{proof}

\begin{Lem}
If $3 \nmid a$,  $3 \nmid b$, $n<m$, and $b < 3^{m-n}$, then
\begin{equation}
\nuthree{3^m a - 3^n b} = 2(m-n)+\nuthree{a}-\nuthree{b}.
\end{equation}
\end{Lem}

\begin{Prop}
\label{prop-case2}
If $1 \le i < \frac{3^n}{2}$, $\muc{3^n+i}= - \muc{2 \cdot 3^n-i+1}$.
\end{Prop}

\begin{proof}
Let $A=3^n+i$ and $B=2 \cdot 3^n - i + 1$. We prove 
 $\muc{A}= - \muc{B}$.

First we observe that 
\begin{eqnarray*}
\muc{A}	
&=& \Sc{2 \cdot 3^n + 2i-1}+\Sc{2 \cdot 3^n + 2i-2}-\Sc{3^{n+1} + 2i-2}-\Sc{3^n + i-1}	\\
&=& (2+\Sc{2i-1}) + (2 + \Sc{2i-2}) - (1+\Sc{3i-2}) - (1+\Sc{i-1} \\
&=& \Sc{2i-1} + \Sc{2i-2} - \Sc{3i-2} - \Sc{i-1} + 2.
\end{eqnarray*}

There are three cases to consider according to the value of $i$ modulo $3$. 
Assume first that $i \equiv 0 \bmod 3$ and write 
$i = 3^a x$, where $a>0$ and $3 \nmid x$. Then 

	\begin{eqnarray*}
	\muc{A}		
	&=& \Sc{2i-1} + \Sc{2i-2} - \Sc{3i-2} - \Sc{i-1} + 2 \\
	&=& \Sc{2 \cdot 3^a x-1} + \Sc{2 \cdot 3^a x-2} - 
                         \Sc{3 \cdot 3^a x-2} - \Sc{3^a x-1} + 2 \\
	&=& (\Sc{2x} -1 + 2a) + (\Sc{2x} -2 + 2a) -  \\
         &  &          (\Sc{x}-2+2(a+1)) - (\Sc{x} - 1 + 2a) + 2 \\
	&=& 2\Sc{2x} - 2\Sc{x}
	\end{eqnarray*}
		
	\begin{eqnarray*}
	\muc{B}	
	&=& \Sc{4 \cdot 3^n -2i+1} + \Sc{4 \cdot 3^n -2i} - 
             \Sc{2 \cdot 3^{n+1} -3i+1} - \Sc{2 \cdot 3^n -i} \\
	&=& \Sc{4 \cdot 3^n -2\cdot 3^a x+1} + 
                             \Sc{4 \cdot 3^n -2\cdot 3^a x} \\
	& &- \Sc{2 \cdot 3^{n+1} -2\cdot 3^{a+1} x+1} - 
                 \Sc{2 \cdot 3^n - 3^a x} \\
	&=& (2n+2-\Sc{2\cdot 3^a x-1}) + (2(n-a)+2-\Sc{2x}) \\
	& &	- (2n+4-\Sc{2\cdot 3^{a+1} x+1}) - (2(n-a)+2-\Sc{x}) \\
	&=& (-\Sc{2x}+1)+ (-\Sc{2 x}) 
		- (-\Sc{2 x}-1) - (-\Sc{x})-2	\\
	&=& -2\Sc{2x} + 2\Sc{x} = -\muc{A},	
	\end{eqnarray*}
\noindent
as claimed. The cases $i \equiv 1, \, 2 \bmod 3$ are analyzed by similar 
techniques. 
\end{proof}

\medskip

\noindent
{\bf Note}. The techniques outlined in this paper can be used to present
a complete description of the function
$\nu_{p}(T(n))$ for $p \geq 5$ prime. We limit ourselves to showing 
the graphs for $p=5$ and $7$ in the range 
$n \leq 5000$. 

\medskip

{{
\begin{figure}[ht]
\begin{center}
\centerline{\epsfig{file=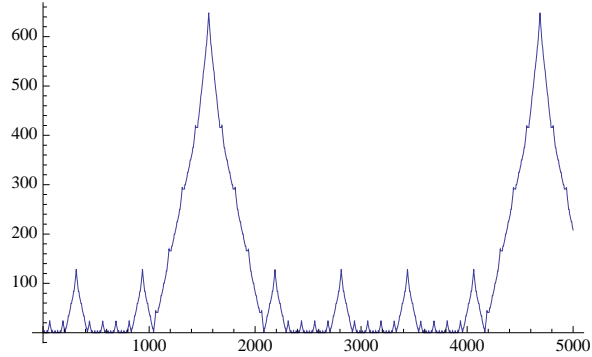,width=20em,angle=0}}
\caption{The $5$-adic valuation of $T(n)$}
\label{figure-5prime}
\end{center}
\end{figure}
}}

\medskip

{{
\begin{figure}[ht]
\begin{center}
\centerline{\epsfig{file=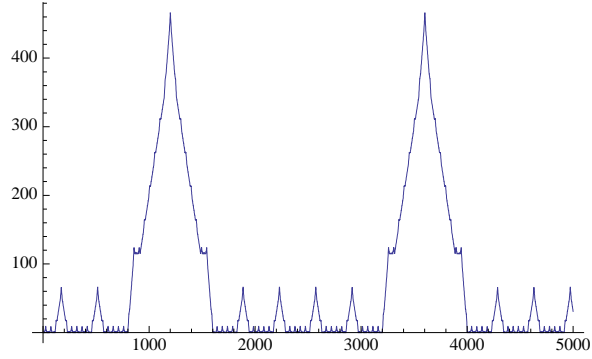,width=20em,angle=0}}
\caption{The $7$-adic valuation of $T(n)$}
\label{figure-7prime}
\end{center}
\end{figure}
}}

\medskip

The rest of the section is devoted to develop an efficient procedure to 
compute $\nu_{3}(T(n))$. We begin with the ternary expansion of $n$
\begin{equation}
n = \sum_{i=0}^{k} a_{i}3^{i},
\end{equation}
\noindent
and now define two sequence of integers: 

\begin{equation}
n_{k}  = n_{k}' = n,
\end{equation}
\noindent
and, for $0 \leq j < k$ and assume having
\begin{equation}
n'_{j+1}=\sum^{j+1}_{i=0}b_{j+1,i} 3^i, 
\end{equation}
\noindent
then define recursively
\begin{eqnarray*}
n_j  &=& \sum^{j}_{i=0}b_{j+1,i} 3^i,	\\
n'_j &=& \left\{
\begin{array}{lll}
n_{j}		& \text{if } b_{j+1,j+1} = 0,2;	\\
\min(n_j, 3^{j+1} - n_{j})	& \text{if } b_{j+1,j+1} = 1.
\end{array}
\right.
\end{eqnarray*}

\begin{Thm}
\label{thm-3adic}
The $3$-adic valuation of $T(n)$ satisfies
\begin{equation}
\nuthree{n_j} = \left\{
\begin{array}{lll}
\nuthree{n'_{j-1}}		& \text{if } a_j = 0,2;	\\
\nuthree{n'_{j-1}} + 2 n'_{j-1}	& \text{if } a_j = 1.
\end{array}
\right.
\end{equation}
\end{Thm}

\noindent
{\bf Note}. Observe that the time 
required to calculate $\nuthree{n}$ is $O(n^2\ln n)$  using the 
definition of $T(n)$. Using Proposition \ref{prop-valuetn} the computational
time reduces to $O(n)$. The method described in Theorem \ref{thm-3adic}
further reduces this time to $O(\ln n)$.  A similar algorithm can be developed
for $p=2$. \\

\noindent
{\bf Example}. Let $n=1280$, whose representation with base 3 is $1202102$. 
Then $k=6$ and we have

\begin{table}[h]
\begin{center}
\begin{tabular}{||c||r|r|r|r||}
\hline
$j$&$n_j$&$n_j$ (base 3)&$n'_j$	&$n'_j$ (base 3)\\\hline
6&1280&1202102&1280&1202102\\
5&551&202102&178&020121	\\
4&178&20121&178&20121\\
3&16&0102&16&0102\\
2&16&102&16&102\\
1&7&21&2&02\\
0&2&2&1&1\\\hline
\end{tabular}
\end{center}
\vskip 0.2in
\caption{The fast algorithm for $\nu_{3} \circ T$}
\end{table}

It follows that
\begin{eqnarray}
\nuthree{1280}  & = & 2 n'_5 + \nuthree{n'_5} \nonumber \\
& = & 2 n'_5+\nuthree{n_2} \nonumber \\
& = & 2 n'_5+2\nuthree{n'_1}+\nuthree{n_1} \nonumber \\
& = & 360. \nonumber 
\end{eqnarray}

\section{A generalization} \label{sec-gen}
\setcounter{equation}{0}

The sequence 
\begin{equation}
T_{p}(n) := \prod_{j=0}^{n-1} \frac{(pj+1)!}{(n+j)!},
\end{equation}
\noindent 
contains $T(n)$ of (\ref{T-seq}) as the special 
case $T(n) = T_{3}(n)$. In 
this section we present some elementary properties of this generalization.

\begin{Thm}
For a fixed prime $p \geq 3$, the numbers $T_{p}(n)$ are integers.
\end{Thm}
\begin{proof}
Observe that
\begin{equation}
T_{p}(n+1) = T_{p}(n) \times \frac{(pn+1)! \, n!}{(2n+1)! \, (2n)!}.
\label{recur-11}
\end{equation}
Define 
\begin{equation}
x_{p}(n) := \frac{(pn+1)!}{((p-1)n+1)! \, n!} = \binom{pn+1}{n},
\end{equation}
\noindent
and observe that
\begin{equation}
\frac{(pn+1)! \, n!}{(2n+1)! \, (2n)!} = x_{p}(n) \times 
\frac{((p-1)n+1)!}{(2n+1)! \, (2n)!}  n!^{2}. 
\end{equation}
\noindent
Iterating this argument yields
\begin{equation}
\frac{(pn+1)! \, n!}{(2n+1)! \, (2n)!} = \prod_{r=0}^{k-1}x_{p-r}(n) 
\times \frac{((p-k)n+1)!}{(2n+1)! \, (2n)!} n!^{k+1}.
\end{equation}
\noindent
The choice  $k = p-4$ confirms that
\begin{equation}
\frac{(pn+1)! \, n!}{(2n+1)! \, (2n)!} = 
\binom{4n+1}{2n} \, n!^{p-3} 
\prod_{r=0}^{p-5}  \binom{(p-r)n+1}{n}
\nonumber
\end{equation}
\noindent
is an integer. The recurrence (\ref{recur-11}) and the initial condition 
$T_{p}(1) = 1$ now show that $T_{p}(n)$ is also an integer. The explicit
formula
\begin{equation}
T_{p}(n) = \prod_{j=1}^{n-1} \binom{4j+1}{2j} j!^{p-3} 
\prod_{r=0}^{p-5} \binom{(p-r)j+1}{j}
\end{equation}
\noindent
follows from the recurrence.  \\
\end{proof}

\begin{proof}
An alternative proof of the fact that 
$\begin{displaystyle} 
\frac{(pn+1)! n!}{(2n+1)! \, (2n)!} \end{displaystyle}$ 
is an integer was shown to us by Valerio de Angelis. Observe that, for 
$p \geq 4$, we have $(pn+1)! = N \times (4n+1)!$ for 
the integer $N = (4n+2)_{(p-4)n}$. Therefore 
\begin{equation}
\frac{(pn+1)! \, n!}{(2n+1)! \, (2n)!} = (4n+2)_{(p-4)n} \times 
\binom{4n+2}{2n} n!.
\end{equation}
\noindent
This leads to the explicit formula
\begin{equation}
T_{p}(n) = \prod_{j=1}^{n-1} (4j+2)_{(p-4)n} \binom{4j+1}{2j} j!.
\end{equation}
\end{proof}

\begin{proof}
A third proof using Theorem \ref{cart-kup} was shown to us by 
T. Amdeberhan. 
The required inequality states:
if $n,k,p \in \mathbb{N}$ and $p\geq 3$, then
$$\psi_k(n;p):=\sum_{j=0}^{n-1}\left\lfloor{{pj+1} \over {k} } \right\rfloor-
\sum_{j=0}^{n-1} \left\lfloor{{n+j} \over {k}} \right\rfloor\geq 0.$$
It suffices to prove the special
case $p=3$, i.e. $\psi_k(n;3)\geq 0$ which we denote by $\psi_k(n)$
for $k\geq 3, n\geq 1$.
\smallskip
\noindent
Write $n=ck+r$ where $0\leq r\leq k-1$. We approach a reduction 
process by breaking down the respective sums as follows. 
\begin{eqnarray}
\sum_{j=0}^{n-1} \left\lfloor{{3j+1} \over{k} } \right\rfloor
 & = & \sum_{j=0}^{ck-1} \left\lfloor{{3j+1} \over {k} } \right \rfloor+
\sum_{j=0}^{r-1} \left\lfloor{{3(ck+j)+1} \over {k} } \right\rfloor \nonumber \\
&= & \sum_{j=0}^{ck-1} \left\lfloor{{3j+1} \over {k} } \right\rfloor+ 3cr
+\sum_{j=0}^{r-1} \left\lfloor{{3j+1} \over {k} } \right\rfloor, \nonumber
\end{eqnarray}
\noindent
and 
\begin{eqnarray}
\sum_{j=0}^{n-1} \left\lfloor{{n+j} \over {k}} \right\rfloor & = & 
\sum_{j=0}^{ck-1} \left\lfloor{{ck + r +j} \over {k}} \right\rfloor
+2cr+\sum_{j=0}^{r-1} \left\lfloor{{r+j} \over {k}} \right \rfloor \nonumber \\
&=& \sum_{j=0}^{ck-1}\left\lfloor{{ck+j} \over{k}} \right\rfloor-
\sum_{j=0}^{r-1} \left\lfloor{{ck+j} \over {k} } \right \rfloor+
\sum_{j=0}^{r-1} \left\lfloor{{2ck+j} \over {k} } \right\rfloor
+2cr+\sum_{j=0}^{r-1} \left\lfloor{{r+j} \over {k} } \right\rfloor \nonumber \\
&=& \sum_{j=0}^{ck-1} \left\lfloor{{ck+j} \over {k} } \right \rfloor+
\sum_{j=0}^{r-1} \left\lfloor{{ck+j} \over {k}} \right\rfloor
+2cr+\sum_{j=0}^{r-1} \left\lfloor{{r+j} \over {k} } \right\rfloor \nonumber \\
&=& \sum_{j=0}^{ck-1} \left\lfloor{{ck+j} \over {k} } \right\rfloor+
cr+\sum_{j=0}^{r-1} \left\lfloor{{j} \over {k} } \right\rfloor
+2cr+\sum_{j=0}^{r-1} \left\lfloor{{r+j} \over {k} } \right\rfloor \nonumber \\
&=&\sum_{j=0}^{ck-1} \left\lfloor{{ck+j} \over {k} } \right\rfloor+
3cr+\sum_{j=0}^{r-1} \left\lfloor{{r+j} \over {k} } \right\rfloor. \nonumber 
\end{eqnarray}
Combining these expressions, we find that $\psi_k(ck+r)=\psi_k(ck)+\psi_k(r)$.
A similar argument with $r$ replaced by $k$ produces 
$\psi_k(ck+k)=\psi_k(ck)+\psi_k(k)$. We conclude $\psi_k$ is 
\it $k$-Euclidean, \rm i.e. 
$$\psi_k(ck+r)=c\psi_k(k)+\psi_k(r).$$
Therefore, we just need to verify the assertion $\psi_k(r)\geq 0$. In 
fact, we will strengthen it by giving an explicit formula in vectorial form
$$[\psi_k(0),\dots,
\psi_k(k-1)]=[0,0^{k^{\prime}},1,2,\dots,\lfloor{k^{\prime\prime}/2}\rfloor,
\lceil{k^{\prime\prime}/2}\rceil,\dots,2,1,0^{k^{\prime}}];$$
where $k^{\prime}=\lfloor{\frac{k+1}3}\rfloor, 
k^{\prime\prime}=k-1-2k^{\prime}$ and $0^{k^{\prime}}$ means $k^{\prime}$ 
consecutive zeros. This admits an elementary proof.  Note that $\psi_k(ck)=0$,
hence $\psi_k$ is \it $k$-periodic \rm and it
satisfies $\psi_k(ck+r)=\psi_k(r)$. 
\end{proof}

\medskip

We now discuss a recurrence for the valuation of the sequence $T_{p}(n)$. The 
special role of the prime $p=3$ becomes apparent. 

\begin{Thm}
Let $p$ be prime. Then the sequence $T_{p}(n)$ satisfies
\begin{equation}
\nu_{p}(T_{p}(pn)) = p \nu_{p}(T_{p}(n)) + \frac{1}{2}p(p-3)n^{2}.
\end{equation}
\end{Thm}
\begin{proof}
Observe  that 
\begin{equation}
T_{p}(pn) = \prod_{j=0}^{pn-1} (pj+1)!/\prod_{j=pn}^{2pn-1} j!
\end{equation}
\noindent
and using Legendre's formula we obtain
\begin{equation}
(p-1) \nu_{p}(T_{p}(pn)) = \sum_{j=0}^{pn-1} pj+1 - S_{p}(pj+1) - 
\sum_{j=pn}^{2pn-1} j-S_{p}(j).
\end{equation}
\noindent
The terms independent of the function $S_{p}$ add up to $n^{2}p(p-3)/2$ and 
we obtain
\begin{equation}
\nu_{p}(T_{p}(pn)) - p \nu_{p}(T_{p}(n)) = 
\frac{1}{2}n^{2}p(p-3) + \frac{1}{p-1}W_{p,n},
\end{equation}
\noindent
where
\begin{equation}
W_{p,n} = - \sum_{j=0}^{pn-1}S_{p}(pj+1) + 
\sum_{j=pn}^{2pn-1} S_{p}(j) + 
p \sum_{j=0}^{n-1}S_{p}(pj+1) - p \sum_{j=0}^{n-1} S_{p}(n+j).
\end{equation}
\noindent
We now show that $W_{p,n} = 0$, this established the result. 

Use $S_{p}(pj+1) = 1 + S_{p}(j)$ to get that
\begin{equation}
W_{p,n} = - \sum_{j=0}^{pn-1}S_{p}(j) + 
\sum_{j=pn}^{2pn-1} S_{p}(j) + 
p \sum_{j=0}^{n-1}S_{p}(j) - p \sum_{j=n}^{2n-1} S_{p}(j).
\end{equation}
\noindent
In the second sum, write $j = pr+k$ with $0 \leq k \leq p-1$ and 
$n \leq r \leq 2n-1$, to obtain
\begin{eqnarray}
\sum_{j=pn}^{2pn-1} S_{p}(j) & = & \sum_{k=0}^{p-1} 
\sum_{r=n}^{2n-1} S_{p}(pr+k) \nonumber \\
& = & \sum_{r=n}^{2n-1} \sum_{k=0}^{p-1} \left( k + S_{p}(r) \right) 
\nonumber \\
& = & \frac{n}{2}p(p-1) + p \sum_{r=n}^{2n-1} S_{p}(r). 
\nonumber 
\end{eqnarray}
\noindent
This term is now combined with the fourth one to simplify the sum. A similar
calculation on the first term gives the result. Indeed,
\begin{eqnarray}
\sum_{j=0}^{pn-1} S_{p}(j) & = & 
\sum_{k=0}^{p-1} \sum_{r=0}^{n-1} S_{p}(pr+k) \nonumber \\
& = & \sum_{k=0}^{p-1} \sum_{r=0}^{n-1} \left( k + S_{p}(r) \right) 
\nonumber \\
& = & \frac{n}{2}p(p-1) + p \sum_{r=0}^{n-1} S_{p}(r). \nonumber 
\end{eqnarray}
\end{proof}

\begin{Cor}
For $p$ a prime, we have
\begin{equation}
\nu_{p}(T_{p}(p^{n})) = \frac{p^{n}(p-3)(p^{n}-1)}{2(p-1)}.
\end{equation}
\end{Cor}
\begin{proof}
Replace $n$ by $p^{n}$ in the Theorem to obtain
\begin{equation}
\nu_{p}(T_{p}(p^{n+1})) = p \nu_{p}(T_{p}(p^{n})) + \frac{1}{2}(p-3)p^{2n+1}.
\end{equation}
\noindent
Iterating this identity yields the result.
\end{proof}

\medskip

\noindent
{\bf Problem}. The sequence $T_{p}(n)$ comes as a formal generalization of 
the original sequence $T_{3}(n)$ that appeared in counting alternating 
symmetric matrices. This begs the question: {\em what do} 
$T_{p}(n)$ {\em count?} 

\bigskip

\no
{\bf Acknowledgments}. 
The authors wish to thank Tewodros Amdeberhan, Valerio de Angelis  and 
A. Straub for many conversations 
about this paper.  Marc Chamberland helped in the experimental 
discovery of the generalization presented in Section \ref{sec-gen}. 
The work of the  first author was partially funded by
$\text{NSF-DMS } 0713836$. \\

\end{document}